\documentclass[]{amsart}

\usepackage[dvipdfmx]{graphicx,color}

\makeatletter
\@addtoreset{equation}{section}

\makeatother

\newtheorem{theorem}{Theorem}[section]

\newtheorem{corollary}[theorem]{Corollary}
\newtheorem{lemma}[theorem]{Lemma}
\theoremstyle{definition}
\newtheorem{definition}[theorem]{Definition}

\theoremstyle{remark}
\newtheorem{remark}[theorem]{Remark}

\newcommand{\e}{\varepsilon}
\newcommand{\be}{\bar{\varepsilon}}

\DeclareMathOperator{\cn}{cn}
\DeclareMathOperator{\dn}{dn}
\DeclareMathOperator{\sn}{sn}

\DeclareMathOperator{\graph}{graph}

\begin{document}

\title{Polar tangential angles and free elasticae}
\author{Tatsuya Miura}
\address{Department of Mathematics, Tokyo Institute of Technology, 2-12-1 Ookayama, Meguro-ku, Tokyo 152-8551, Japan}
\email{miura@math.titech.ac.jp}
\date{\today}
\subjclass[2010]{}

\begin{abstract}
In this note we investigate the behavior of the polar tangential angle of a general plane curve, and in particular prove its monotonicity for certain curves of monotone curvature.
As an application we give (non)existence results for an obstacle problem involving free elasticae.
\end{abstract}

\maketitle


\section{Introduction}\label{sectintroduction}

The purpose of this note is to develop a geometric approach to elastic curve problems, i.e., variational problems involving the total squared curvature, also known as the bending energy.
The variational study of elastic curves is originated with D.~Bernoulli and L.~Euler in the 1740's, but it is still ongoing even concerning the original (clamped) boundary value problem.
In particular, the properties of solutions such as uniqueness or stability are not fully understood and sensitively depend on the parameters in the constraints; see e.g.\ \cite{Linner1998a,Singer2008,Sachkov2008,Miura2020} and references therein.

In order to study boundary conditions involving the tangent vector, it would be helpful to precisely understand the so-called polar tangential angle for a plane curve, which is the angle formed between the position vector and the tangent vector.
In this note we give a geometric characterization of the first-derivative sign of the polar tangential angle, and then deduce the monotonicity of the angle for certain curves of monotone curvature with the help of the classical Tait-Kneser theorem.

Our geometric aspect would be useful for studying an elastic curve, since its curvature is represented by Jacobi elliptic functions, whose monotone parts are well understood.
As a concrete example, we apply our monotonicity result to a free boundary problem involving the bending energy.
More precisely, we minimize the bending energy among graphical curves $u:[-1,1]\to\mathbb{R}$ subject to the boundary condition $u(\pm1)=0$ such that $u\geq\psi$, where $\psi$ is a given obstacle function.
The existence of graph minimizers is a somewhat delicate issue \cite{DallAcqua2018,Mueller2019,Yoshizawa2019} in contrast to the confinement-type obstacle problem in \cite{Dayrens2018}.
In this paper, choosing $\psi$ to be a symmetric cone, we prove (non)existence results depending on the height of the cone.
Our results reprove the nonexistence result of M\"{u}ller \cite{Mueller2019} by a novel geometric approach, and also provide a new uniqueness result in the class of symmetric graphs.
Very recently, the same uniqueness result is independently obtained by Yoshizawa \cite{Yoshizawa2019} in a different way; see Remark \ref{rem:obstacle} for a precise comparative review.

This paper is organized as follows.
Section \ref{sec:polar} is devoted to understanding the polar tangential angle.
In Section \ref{sec:obstacle} we apply a monotonicity result (Corollary \ref{cor:monotoneangle}) to the aforementioned obstacle problem (Theorem \ref{thm:obstacle}).

\subsection*{Acknowledgments}
When attending a mini-symposium in the OIST hosted by \mbox{James McCoy}, the author was informed by \mbox{Kensuke Yoshizawa} that he has also obtained the same kind of results \cite{Yoshizawa2019}.
The author would like to thank them for encouraging publication of this paper.
The author is also grateful to \mbox{Shinya Okabe}, \mbox{Glen Wheeler}, and also an anonymous referee for giving helpful comments to an earlier version of this manuscript.
This work is in part supported by JSPS KAKENHI Grant Number 18H03670 and 20K14341, and by Grant for Basic Science Research Projects from The Sumitomo Foundation.

\section{Geometry of polar tangential angles}\label{sec:polar}

Throughout this section we consider a smooth plane curve $\gamma$ parameterized by the arclength parameter $s$, namely, $\gamma\in C^\infty([0,L),\mathbb{R}^2)$, where $L\in(0,\infty]$, and $|\gamma_s(s)|=1$ for any $s$, where the subscript of $s$ denotes the arclength derivative.
Let $T$ denote the unit tangent $\gamma_s$, and $N:=R_{\pi/2}T$ the unit normal, where $R_\theta$ stands for the counterclockwise rotation matrix through angle $\theta\in\mathbb{R}$.

Our main object of study is the polar tangential angle.
To define it rigorously, we call a plane curve $\gamma$ {\em generic} if the curve does not pass the origin $O\in\mathbb{R}^2$ except for $s=0$.
For such a curve we denote the normalized position vector by $X:=\gamma/|\gamma|$.

\begin{definition}[Polar tangential angle function]
  For a generic plane curve $\gamma:[0,L)\to\mathbb{R}^2$, the {\em polar tangential angle function} $\omega:(0,L)\to\mathbb{R}$ is defined as a smooth function such that $R_{\omega}X=T$ holds on $(0,L)$.
\end{definition}

\begin{figure}
  \includegraphics[width=25mm]{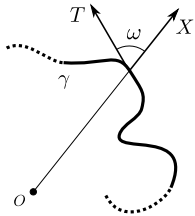}
  \caption{Polar tangential angle.}
  \label{fig:polartangentialangle}
\end{figure}

\begin{remark}
  The value of $\omega$ coincides with the angle between $X$ and $T$, i.e., $\arccos(X\cdot T)$, where $\cdot$ denotes the inner product, as long as $\omega\in[0,\pi]$; see Figure \ref{fig:polartangentialangle}.
  The polar tangential angle function $\omega$ is smooth everywhere and unique up to addition by a constant in $2\pi\mathbb{Z}$.
  Unless a curve is generic, the function $\omega$ needs to be discontinuous.
\end{remark}

The polar tangential angle is a classical notion and often used in the literature.
For example, the logarithmic spiral ($r=e^{a\theta}$ in the polar coordinates) is also known as the {\em equiangular spiral} since its polar tangential angle is constant.
In this section we gain more insight into the behavior of the polar tangential angle.

\subsection{A general derivative formula for the polar tangential angle}

We first give a general formula for the derivative of the polar tangential angle function $\omega$.
Let $\kappa$ denote the {\em (signed) curvature}, i.e., $\gamma_{ss}=\kappa N$.
Recall that the {\em evolute} $\e$ of a plane curve $\gamma$ is defined as the locus of the centers of osculating circles, that is, for points where $\kappa\neq0$,
$$\e:=\gamma+\kappa^{-1}N.$$
Furthermore, in order to state our theorem in a unified way, we introduce 
$$\be:=\kappa\gamma+N.$$
This is same as $\kappa\e$ with the understanding that $\kappa\e=N$ when $\kappa=0$, and in particular defined everywhere as opposed to the original evolute.
Obviously, the direction of $\be$ is same as (resp.\ opposite to) that of $\e$ if $\kappa>0$ (resp.\ $\kappa<0$).

Here is the key identity for the polar tangential angle.

\begin{theorem}\label{thmcharacterization}
  The identity $\omega_s=|\gamma|^{-2}(\gamma\cdot\be)$ holds for any generic plane curve $\gamma$.
\end{theorem}

\begin{proof}
  Since $N_s=-\kappa T$ and also $X_s=|\gamma|^{-1}[T-(X\cdot T)X]$, we have
  \begin{align*}
    (X\cdot N)_s= X_s\cdot N+X\cdot N_s 
    =-(X\cdot T)[(\gamma\cdot\be)/|\gamma|^2].
  \end{align*}
  Inserting $X\cdot T=\cos\omega$ and $X\cdot N=-\sin\omega$ to the above identity, we obtain
  \begin{align}\label{eqn:mainthm1}
    (-\cos\omega)\omega_s=(-\cos\omega)[(\gamma\cdot\be)/|\gamma|^2],
  \end{align}
  which ensures the assertion as long as $\cos\omega\neq0$.
  Similarly, we also have
  \begin{align*}
    (X\cdot T)_s =|\gamma|^{-1}(1-(X\cdot T)^2)+\kappa(X\cdot N) 
    =(X\cdot N)[(\gamma\cdot\be)/|\gamma|^2],
  \end{align*}
  where the identity $1-(X\cdot T)^2=(X\cdot N)^2$ is used, and hence
  \begin{align}\label{eqn:mainthm2}
    (-\sin\omega)\omega_s=(-\sin\omega)[(\gamma\cdot\be)/|\gamma|^2].
  \end{align}
  Combining (\ref{eqn:mainthm1}) and (\ref{eqn:mainthm2}), we complete the proof.
\end{proof}

As Theorem \ref{thmcharacterization} is not geometrically intuitive, we now characterize the first-derivative sign of the polar tangential angle from a geometric point of view.
To this end we separately consider two cases, depending on whether the curvature vanishes.

We first think of the simpler case that the curvature vanishes.
For a generic plane curve $\gamma$ we let $L(s)$ denote the half line from the origin and in the same direction as the vector $\gamma(s)$.
Then we have the following geometric characterization, the proof of which is safely omitted; see Figure \ref{fig:sign1} and Theorem \ref{thmcharacterization}.

\begin{corollary}
  \label{cor:sign1}
  Let $\gamma$ be a generic plane curve and assume that $\kappa(s)=0$ at some $s$.
  Then the signs of $\omega_s(s)$ and $\gamma(s)\cdot N(s)$ coincide.
  In other words, $\omega_s(s)$ is positive (resp.\ negative) if and only if the curve $\gamma$ transversally passes $L(s)$ from left to right (resp.\ right to left) at $s$.
  In addition, $\omega_s(s)=0$ if and only if $\gamma$ touches $L(s)$ at $s$.
\end{corollary}

\begin{figure}
  \includegraphics[width=65mm]{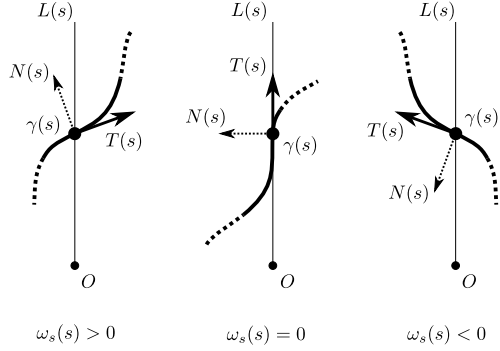}
  \caption{The sign of $\omega_s$ when $\kappa=0$.}
  \label{fig:sign1}
\end{figure}

We now turn to the more interesting case that the curvature does not vanish.
Since $\omega$ is an oriented notion, it is suitable to consider the sign of $\kappa\omega_s$ rather than $\omega_s$.
For a given generic curve, we let $C(s)$ denote the osculating circle at $s$, namely, the circle of radius $1/|\kappa(s)|$ centered at $\e(s)$ ($=\gamma(s)+(\kappa(s))^{-1}N(s)$) provided that $\kappa(s)\neq0$.
In addition, let $\hat{c}(s)$ be the circle whose diameter is attained by the points $\gamma(s)$ and $\e(s)$.
Then we have the following geometric characterization, the proof of which is again safely omitted; see Figure \ref{fig:sign2} and Theorem \ref{thmcharacterization}.

\begin{corollary}
  \label{cor:sign2}
  Let $\gamma$ be a generic plane curve and assume that $\kappa(s)\neq0$ at some $s$.
  Then the signs of $\kappa(s)\omega_s(s)$ and $\gamma(s)\cdot\e(s)$ coincide.
  In other words, $\kappa(s)\omega_s(s)$ is positive (resp.\ negative) if and only if the origin is outside (resp.\ inside) the circle $\hat{c}(s)$.
  In addition, $\kappa(s)\omega_s(s)=0$ if and only if the origin lies on $\hat{c}(s)$.
  In particular, $\kappa\omega_s$ is positive as long as the osculating circle $C$ does not enclose the origin.
\end{corollary}

\begin{figure}
  \includegraphics[width=90mm]{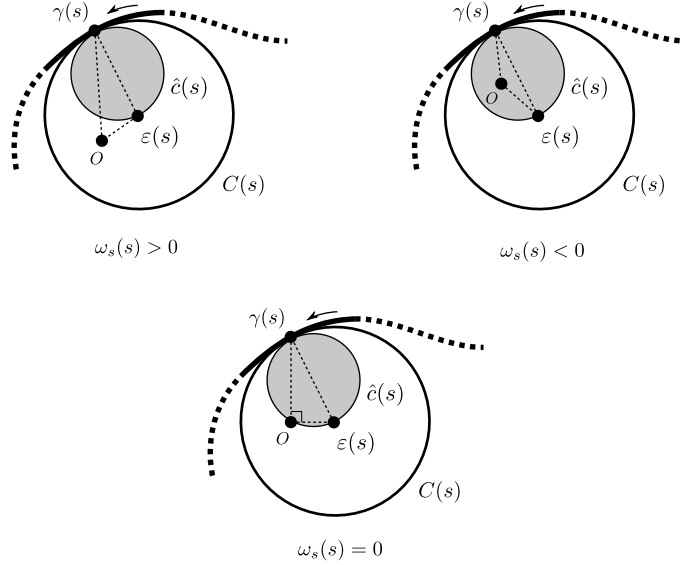}
  \caption{The sign of $\omega_s$ when $\kappa>0$.}
  \label{fig:sign2}
\end{figure}

\subsection{Sufficient conditions for monotonicity}

In this subsection we provide simple sufficient conditions for the monotonicity of the polar tangential angle.

The key assumption is monotonicity of the curvature, which controls the global behavior of osculating circles.
Indeed, the classical {\em Tait-Kneser theorem} \cite{Tait1896,Kneser1912} (see also \cite{Ghys2013}) states that if the curvature of a plane curve is strictly monotone, then the osculating circles are pairwise disjoint.
In addition, if the curvature has no sign change, then subsequent circles are nested.
Moreover, it is not difficult to observe that if the curvature and its derivative have same sign (resp.\ different sign), then the osculating circles become smaller (resp.\ larger) as the arclength parameter increases.

The monotonicity of the polar tangential angle is somewhat delicate, and in fact the monotonicity of curvature is still not sufficient.
Here we impose an additional assumption on the initial state $s=0$.
Let $D(s)$ denote the open disk enclosed by the osculating circle $C(s)$.
For a generic plane curve $\gamma$ such that $(\kappa^2)_s=2\kappa\kappa_s>0$ for any $s>0$,
we define the {\em initial osculating disk} $\widetilde{D}(0)$ as the open set given by
$$\widetilde{D}(0):=\lim_{s\downarrow0}D(s),$$
where the limit is well defined since $0<s_1<s_2\Rightarrow D(s_1)\subset D(s_2)$ thanks to the Tait-Kneser theorem and the discussion above.
Notice that if $\kappa(0)\neq0$, then $\widetilde{D}(0)$ is nothing but the open disk $D(0)$ enclosed by the osculating circle $C(0)$.
If $\kappa(0)=0$, then $\widetilde{D}(0)$ is a limit half-plane:
$$\widetilde{D}(0)=\gamma(0)+\{ p\in\mathbb{R}^2 \mid p\cdot \widetilde{N}(0)>0 \},\quad \mbox{where}\  \widetilde{N}(0):=\lim_{s\downarrow0}\gamma_{ss}(s)/|\gamma_{ss}(s)|.$$
Then we have the following

\begin{corollary}\label{cor:monotone1}
  Let $\gamma$ be a generic plane curve.
  If the derivative of $\kappa^2$ is positive at any $s>0$, and if the initial osculating disk $\widetilde{D}(0)$ does not include the origin $O$, then $\kappa(s)\omega_s(s)>0$ for any $s>0$.
  In particular, $\omega$ is strictly monotone.
\end{corollary}

\begin{proof}\label{cor:monotone2}
  By the Tait-Kneser theorem and by the positivity of the derivative of $\kappa^2$, we have $\overline{D(s_1)}\subset D(s_0)$ for any $0<s_0<s_1$, where $\overline{D(s_1)}$ denotes the closure of $D(s_1)$.
  Since $D(s_0)\subset\widetilde{D}(0)$ holds by definition of $\widetilde{D}(0)$, the assumption $O\not\in\widetilde{D}(0)$ implies that $O\not\in\overline{D(s)}$ for any $s>0$.
  Therefore, Corollary \ref{cor:sign2} implies that $\kappa\omega_s>0$ for any $s>0$.
  Monotonicity of $\omega$ follows immediately.
\end{proof}

In particular, if such a curve $\gamma$ starts from the origin, then the condition $O\not\in\widetilde{D}(0)$ is obviously satisfied, and also $\gamma(s)\neq0$ always holds for $s>0$ thanks to the Tait-Kneser theorem.
We conclude with an immediate consequence of Corollary \ref{cor:monotone1}, the proof of which can be safely omitted.

\begin{corollary}\label{cor:monotoneangle}
  Let $\gamma:[0,L]\to\mathbb{R}^2$ be a plane curve such that $\gamma(0)=0$.
  If $\kappa\kappa_s$ is positive on $(0,L)$, then so is $\kappa\omega_s$.
  In particular, $\omega$ is strictly monotone.
\end{corollary}

\section{Application to an obstacle problem for free elasticae}\label{sec:obstacle}

We apply our monotonicity result to the following higher order obstacle problem:
\begin{equation}\label{eq:min}
  \inf_{u\in X_\psi} B[u], \qquad \mbox{where}\ B[u]:= \int_{\graph{u}}\kappa^2ds,
\end{equation}
the admissible function space is given by
$$X_\psi=\{u\in W^{2,2}(I)\mid u(\pm1)=0,\ u\geq\psi\}\quad \mbox{with}\ I=(-1,1),$$
and $\psi\in C(\bar{I})$ is a symmetric cone such that $\psi(\pm1)<0$.
Here we call $\psi$ a symmetric cone if $\psi(x)=\psi(-x)$ and $\psi$ is affine on $(0,1)$.
The functional $B[u]$ means the total squared curvature (also known as the bending energy) along the graph curve of $u$.
For a graphical curve, $B$ can be expressed purely in terms of the height function $u$ via the formula: $B[u]=\int_Iu''^2(1+u'^2)^{-5/2}dx$.

Here we prove that Corollary \ref{cor:monotoneangle} implies certain (non)existence results.
Let
\begin{equation}\label{eq:height}
  c_*:=\int_{0}^{\pi/2}\sqrt{\cos\varphi}d\varphi, \quad h_*:=\frac{2}{c_*} = 1.66925368...
\end{equation}
Let $X_{\psi,\mathrm{sym}}$ be the subspace of all even symmetric functions in $X_\psi$.

\begin{theorem}\label{thm:obstacle}
  Let $\psi\in C(\bar{I})$ be a symmetric cone such that $\psi(\pm1)<0$ and $\psi(0)=:h>0$.
  If $h<h_*$, then there is a unique minimizer $\bar{u}\in X_{\psi,\mathrm{sym}}$ of $B$ in $X_{\psi,\mathrm{sym}}$.
  If $h\geq h_*$, then there is no minimizer of $B$ in $X_{\psi,\mathrm{sym}}$, and also in $X_\psi$.
\end{theorem}

\begin{remark}[Representability of unique solutions]
  Our proof of Theorem \ref{thm:obstacle} (or more precisely Lemma \ref{lem:1} below) immediately implies that the curvature as well as the angle function of our unique symmetric solution for $h\in(0,h_*)$ can be explicitly parameterized in terms of elliptic functions.
  Note however that to this end we need to use some constants uniquely characterized by $h$, respectively, which solve somewhat complicated transcendental equations involving elliptic functions.
  In this paper we do not go into the details of completely explicit formulae.
\end{remark}

\begin{remark}[Non-graphical case]
  If we allow non-graphical curves to be competitors, then there is no minimizer because an arbitrary large circular arc circumventing the obstacle is admissible so that the infimum is zero, but this infimum is not attained as a straight segment is not admissible due to the obstacle.
  Thus our graphical minimizers may be regarded as nontrivial critical points in the non-graphical problem.
\end{remark}

\begin{remark}[Comparative review]\label{rem:obstacle}
  The existence of a minimizer in the symmetric class is already obtained by Dall'Acqua-Deckelnick \cite{DallAcqua2018} for a more general $\psi$, and hence the novel part is the uniqueness result.
  The nonexistence result is obtained by M\"{u}ller \cite{Mueller2019} except for the critical value $h=h_*$.
  In addition, very recently, Yoshizawa \cite{Yoshizawa2019} independently obtained the same results as in Theorem \ref{thm:obstacle} by a different approach, which is based on a shooting method and directly deals with a fourth order ODE for $u$.
  Our method is more geometric and mainly focuses on the curvature, thus being significantly different from the previous methods \cite{DallAcqua2018,Mueller2019,Yoshizawa2019}.
  We expect that our geometric aspect is also useful for analyzing critical points of other functionals, e.g., including the effect of length, or dealing with non-quadratic exponents, since in both cases the curvature monotone part of a solution is well understood.
  However, the presence of length yields a multiplier so that the elliptic modulus is not fixed \cite{Linner1998}, while a large exponent $p>2$ yields a so-called ``flat-core'' solution involving interval-type zeroes of curvature \cite{Watanabe2014}, and hence in both cases more candidates of solutions need to be considered.
  We finally mention that the aforementioned authors in \cite{DallAcqua2018,Mueller2019} use a different scaling, namely $I=(0,1)$ instead of $I=(-1,1)$, but it is easy to see that their value of $h_*$ is consistent with ours.
\end{remark}

\subsection{Free elastica}\label{subsec:freeelastica}

For later use, as well as clarifying the reason why $h_*$ appears in Theorem \ref{thm:obstacle}, we recall some well-known facts about minimizers.

We first recall that if $u\in X_{\psi}$ is a minimizer in \eqref{eq:min}, then we have
\begin{align}\label{eq:BVP}
  \begin{cases}
    u\ \mbox{is concave},\\
    2\kappa_{ss}+\kappa^3=0 \quad\mbox{on}\ \graph{u}\setminus\graph{\psi},\\
    u(\pm1)=u''(\pm1)=0.
  \end{cases}
\end{align}
In fact, the concavity follows since otherwise taking the concave envelope decreases the energy, cf.\ \cite{DallAcqua2018}; the equation in the second line and the last boundary condition follow by standard calculation of the first variation.
We remark that the equation is understood first in the sense of distribution, but then in the classical sense by using a standard bootstrapping argument, cf.\ \cite{Dayrens2018}.
The second order boundary condition means the curvature of the graph vanishes at the endpoints.
Notice that by concavity of $u$ we immediately deduce that the coincidence set $\graph{u}\cap\graph{\psi}$ is either empty or the apex of the cone; we will see later that it cannot be empty.

A solution to the equation in the second line of \eqref{eq:BVP} is called a {\em free elastica}, which is a specific example of Euler's elastica.
This equation possesses the fine scale invariance in the sense that if a curve $\gamma$ is a solution, then so is every curve that coincides with $\gamma$ up to similarity (i.e., translation and scaling).
A free elastica is essentially unique and described in terms of the Jacobi elliptic function.

Given $k\in[0,1]$, we let $\cn(x;k)$ denote the elliptic cosine function with elliptic modulus $k$, that is, $\cn(x;k):=\cos\phi$ for a unique value $\phi$ such that $x=\int_0^\phi g(\theta)d\theta$, where $g(\theta)=(1-k^2\sin^2\theta)^{-1/2}$.
Recall that $\cn(x;k)$ is $4K(k)$-periodic and symmetric in the sense that $\cn(x;k)=\cn(-x;k)=-\cn(x+2K(k);k)$, and $\cn(x;k)$ strictly decreases from $1$ to $0$ as $x$ varies from $0$ to $K(k)$, where $K(k)$ denotes the complete elliptic integral of the first kind $\int_0^{\pi/2}g(\theta)d\theta$.
In addition, the elliptic sine function is similarly defined by $\sn(x;k):=\sin\phi$ by using the above $\phi$, and also the delta amplitude by $\dn(x;k):=(1-k^2\sin^2\phi)^{1/2}$.
For later use we recall the basic formulae: $\sn^2+\cn^2=1$, $\frac{d}{dx}\sn=\cn\dn$, and $\frac{d}{dx}\cn=-\sn\dn$.

It is known (cf.\ \cite[Proposition 2.3]{Linner1993}) that any solution to $2\kappa_{ss}+\kappa^3=0$ is of the form
\begin{equation}\label{eq04}
  \kappa(s)=\sqrt{2}\lambda\cn(\lambda s+\mu;\tfrac{1}{\sqrt{2}}), \quad \mbox{where}\ \lambda,\mu\in\mathbb{R}.
\end{equation}
If $\lambda=0$, then the solution is a trivial straight segment, while if $\lambda\neq0$, then the solution curve is called a rectangular elastica.
Since $\lambda$ is a scaling factor and $\mu$ is just shifting the variable, nontrivial solutions are essentially unique.

As a key fact, up to similarity, a rectangular elastica is represented by a part of the graph curve of a periodic function $U$ as in Figure \ref{fig:freeelastica}.

\begin{figure}
  \includegraphics[width=60mm]{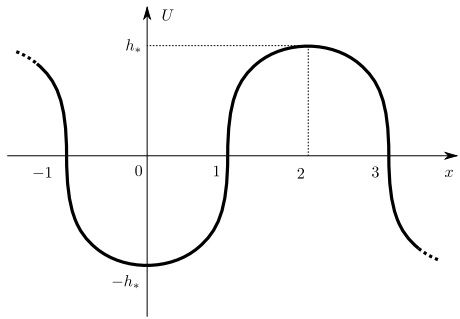}
  \caption{The graph of $U$: A rectangular elastica.}
  \label{fig:freeelastica}
\end{figure}

\begin{lemma}[Graph representation of a rectangular elastica]\label{lem:graph}
  Let $\gamma:\mathbb{R}\to\mathbb{R}^2$ be a smooth plane curve parameterized by the arclength such that $\gamma(0)=(0,-h_*)$, $\gamma_s(0)=(1,0)$, and the curvature is given by \eqref{eq04} with $\lambda=c_*/\sqrt{2}$ and $\mu=0$, where $c_*$ and $h_*$ are defined in \eqref{eq:height}.
  Then $\gamma$ is represented by the graph curve of a function $U:\mathbb{R}\to\mathbb{R}$ satisfying the following properties:
  \begin{itemize}
    \item $U(x)=U(-x)=-U(x+2)$ (and hence $U|_{[0,1]}$ determines the whole shape),
    \item $U$ is smooth in $(-1,1)$ while having vertical slope at $x=\pm1$,
    \item $U$ takes the minimum $-h_*$ at $x=0$,
    \item the curvature $\kappa$ of $\graph{U}$ is positive for $x\in[0,1)$ and vanishes for $x=1$,
    \item the arclength derivative $\kappa_s$ vanishes for $x=0$ and is negative for $x\in(0,1]$.
  \end{itemize}
  In particular, every plane curve with curvature of the form \eqref{eq04} coincides with a part of the graph curve of $U$ up to similarity.
\end{lemma}

Although the graph representability of a rectangular elastica is a classical fact, we provide here a complete proof that only relies on the curvature representation in terms of elliptic functions for the reader's convenience.

\begin{proof}[Proof of Lemma \ref{lem:graph}]
  We first mention that the last statement immediately follows by the aforementioned fact that the parameters $\lambda$ and $\mu$ correspond to scaling and parameter-shifting, respectively, and by the elementary fact that a general plane curve is characterized by the curvature function up to translation and rotation.

  We now prove the graph representability.
  For computational simplicity, we mainly investigate the behavior of a curve $\gamma=(x,y)$ such that the curvature is given by \eqref{eq04} with $\lambda=1$ and $\mu=0$, namely
  \begin{equation*}
    \kappa(s)=\sqrt{2}\cn(s;\tfrac{1}{\sqrt{2}}),
  \end{equation*}
  and later rescale the parameter $\lambda$.
  By the periodicity of $\cn$, we only need to focus on the part in the quarter period $[0,K(\tfrac{1}{\sqrt{2}})]$ (corresponding to the graph curve of $U|_{[0,1]}$ up to similarity).
  We first notice that, since the primitive function of $\cn(t;k)$ is $\frac{1}{k}\arcsin(k\sn(t;k))$, by normalizing $\theta(0)=0$, we can represent the angle function $\theta$ of $\gamma$ by
  \begin{equation}\label{eq01}
    \theta(s) = \int_0^s\kappa(s')ds' = 2\arcsin\left(\tfrac{1}{\sqrt{2}}\sn\left(s;\tfrac{1}{\sqrt{2}}\right)\right).
  \end{equation}
  In particular, we have
  \begin{equation}\label{eq02}
    \theta\big(K(\tfrac{1}{\sqrt{2}})\big)-\theta(0)=\frac{\pi}{2}.
  \end{equation}
  In addition, using \eqref{eq01}, and noting that $0\leq\theta(s)\leq\pi/2$ in the quarter period, we obtain the following representations for $s\in[0,K(\tfrac{1}{\sqrt{2}})]$:
  \begin{align*}
    y(s)-y(0) &= \int_0^s \sin\theta(t)dt= \int_0^s 2\sin\tfrac{\theta(t)}{2}\sqrt{1-\sin^2\tfrac{\theta(t)}{2}}dt\\
    &= \int_0^s \sqrt{2}\sn(t;\tfrac{1}{\sqrt{2}})\dn(t;\tfrac{1}{\sqrt{2}}) dt = \sqrt{2}\left(1-\cn(s;\tfrac{1}{\sqrt{2}})\right),\\
    x(s)-x(0) &= \int_0^s \cos\theta(t)dt = \int_0^s (1-2\sin^2\tfrac{\theta(t)}{2}) dt = \int_0^s (1-\sn^2(t;\tfrac{1}{\sqrt{2}})) dt\\
    &= \int_0^s \tfrac{d}{dt}\Big(F\big(\sn(t;\tfrac{1}{\sqrt{2}})\big)\Big) dt = F\big(\sn(s;\tfrac{1}{\sqrt{2}})\big),
  \end{align*}
  where $F(r):=\int_0^r(1-\sigma^2)^{1/2}(1-\sigma^2/2)^{-1/2}d\sigma$.
  In particular, we immediately have
  \begin{equation}\label{eq08}
    y\big(K(\tfrac{1}{\sqrt{2}})\big)-y(0)=\sqrt{2},
  \end{equation}
  and also we deduce that
  \begin{equation}\label{eq09}
    x(K\big(\tfrac{1}{\sqrt{2}})\big)-x(0)=F(1)=\frac{c_*}{\sqrt{2}},
  \end{equation}
  where the last calculation follows by the change of variables $\sigma=\sqrt{2}\sin(\varphi/2)$, cf.\ \eqref{eq:height}.
  Combining \eqref{eq08} and \eqref{eq09}, we in particular have
  \begin{equation}\label{eq03}
    y\big(K(\tfrac{1}{\sqrt{2}})\big)-y(0)=h_*\Big(x(K\big(\tfrac{1}{\sqrt{2}})\big)-x(0)\Big).
  \end{equation}
  We now rescale $\lambda$ to be $c_*/\sqrt{2}$ so that $x(K\big(\tfrac{1}{\sqrt{2}})\big)-x(0)=1$, cf.\ \eqref{eq09}.
  Then, since \eqref{eq02} and \eqref{eq03} are scale invariant, we can easily check that a curve $\gamma$ with $\kappa(s)=c_*\cn(\tfrac{c_*}{\sqrt{2}}s;\tfrac{1}{\sqrt{2}})$ defines the desired function $U$.
\end{proof}

From the representation of a free elastica, we can now deduce that the coincidence set in \eqref{eq:BVP} is not empty, since otherwise the graph curve of $u$ would be fully a free elastica that satisfies the vanishing-curvature boundary condition, but this contradicts the fact that $u$ cannot have a vertical slope (as $W^{2,2}(I)\subset C^1(\bar{I})$).
Therefore, with the help of concavity we have
\begin{equation}\label{eq:apex}
  \graph{u}\cap\graph{\psi}=\{(0,\psi(0))\}.
\end{equation}
This in particular means that a minimizer $u$ is smooth except at the origin.

\subsection{Boundary value problems}\label{subsec:BVP}

Keeping the facts in Section \ref{subsec:freeelastica} in mind, we now turn to the proof of Theorem \ref{thm:obstacle}.
To this end, given a positive constant $h>0$, we consider the following boundary value problem for a smooth function $u$ on $[0,1]$:
\begin{align}\label{eq:BVPsym}
  \begin{cases}
    2\kappa_{ss}+\kappa^3=0,\\
    u(0)=u''(0)=0,\ u(1)=h,\ u'(1)=0,
  \end{cases}
\end{align}
where the first equation is solved by the whole graph curve of $u$.

\begin{lemma}\label{lem:1}
  If $h\geq h_*$, then there is no solution to \eqref{eq:BVPsym}.
  If $h<h_*$, then there exists a unique solution to \eqref{eq:BVPsym}.
\end{lemma}

\begin{proof}
  Define $U_*:[0,1]\to[0,h_*]$ by $U_*(x):=U(x+1)$, where $U$ is defined in Lemma \ref{lem:graph}.
  Notice that $U_*$ is increasing and concave, and has vertical slope and vanishing curvature at $x=0$.
  By the uniqueness property of free elasticae in Lemma \ref{lem:graph}, a smooth function $u:[0,1]\to\mathbb{R}$ satisfies \eqref{eq:BVPsym} if and only if
  \begin{equation}\label{eq05}
    \begin{cases}
      \exists\lambda>0,\ \exists\phi\in(-\pi/2,0)\ \mbox{such that}\ \graph{u}\subset \lambda R_\phi\graph{U_*},\\
      u(1)=h,\ u'(1)=0,
    \end{cases}
  \end{equation}
  cf.\ Figure \ref{fig:BVP}, where $R_\phi$ is the counterclockwise rotation matrix through $\phi$.

  We now confirm that finding a solution $u$ to \eqref{eq05} is equivalent to the following problem:
  \begin{equation}\label{eq:reduced2}
    \begin{split}
      &\mbox{Find a pair $(\lambda,\phi)\in(0,\infty)\times(-\pi/2,0)$ such that}\\
      &\mbox{$\lambda R_\phi \graph{U_*}$ and the line $\{y=h\}$ meet tangentially at $(1,h)\in\mathbb{R}^2$.}
    \end{split}
  \end{equation}
  Given a solution $u$ to \eqref{eq05}, we let $(\lambda_u,\phi_u)$ denote a corresponding pair in \eqref{eq05}.
  This pair is in fact uniquely characterized by $u$; indeed, the value of $u'(0)$ first characterizes the angle $\phi_u$ to be $\arctan{u'(0)}-\pi/2$ since $U_*'(0+)=\infty$; then $\lambda_u$ is also characterized, since $U_*$ is concave and $U_*(0)=0$ so that $\graph{U_*}$ and a half-line $\{\Lambda e\mid \Lambda>0\}$, $e\in\mathbb{S}^1$, intersect (at most) one point.
  Clearly, such $(\lambda_u,\phi_u)$ solves \eqref{eq:reduced2}, and the correspondence of solutions $u \mapsto (\lambda_u,\phi_u)$ is one-to-one.

  We further prove that solving \eqref{eq:reduced2} is also equivalent to the following problem:
  \begin{equation}\label{eq:reduced}
    \mbox{Find $s\in(0,L_*)$ such that $\tan(-\omega_*(s))=h$},
  \end{equation}
  where $L_*$ and $\omega_*$ denote the length and the polar tangential angle of the unit-speed curve $\gamma_*$ representing $\graph{U_*}$, respectively.
  Given a solution $(\lambda,\phi)$ to \eqref{eq:reduced2}, we define $s_{(\lambda,\phi)}\in(0,L_*)$ so that
  \begin{equation}\label{eq07}
    \lambda R_{\phi}\gamma_*(s_{(\lambda,\phi)})=(1,h).
  \end{equation}
  Clearly, the map $(\lambda,\phi)\mapsto s_{(\lambda,\phi)}$ is well defined.
  Thanks to invariance of $\omega_*$ with respect to rotation and scaling, this $s_{(\lambda,\phi)}$ solves \eqref{eq:reduced}, cf.\ Figure \ref{fig:BVP}; in particular,
  \begin{equation}\label{eq06}
    R_{\phi}T_*(s_{(\lambda,\phi)})=(1,0), \quad \mbox{where } T_*:=(\gamma_*)_s.
  \end{equation}
  We now prove that the correspondence of solutions $(\lambda,\phi)\mapsto s_{(\lambda,\phi)}$ is one-to-one.
  The injectivity follows since if $s_{(\lambda,\phi)}=s_{(\lambda',\phi')}$ holds for two solutions $(\lambda,\phi)$ and $(\lambda',\phi')$ to \eqref{eq:reduced2}, then $\phi=\phi'$ holds by \eqref{eq06}, and moreover $\lambda=\lambda'$ holds by \eqref{eq07}.
  Concerning the surjectivity, for any $s\in(0,L_*)$ solving \eqref{eq:reduced}, we can choose $\phi\in(-\pi/2,0)$ satisfying $R_\phi T_*(s)=(1,0)$ thanks to the shape of $U_*$, and also choose $\lambda>0$ satisfying $\lambda R_\phi\gamma_*(s)=(1,h)$ thanks to $\tan(-\omega_*(s))=h$, so that the resulting curve $\lambda R_\phi\graph{U_*}$ solves \eqref{eq:reduced2}, cf.\ Figure \ref{fig:BVP}.

  Consequently, the number of solutions to \eqref{eq:BVPsym} is characterized by that of \eqref{eq:reduced}.
  We finally investigate how the number of solutions of \eqref{eq:reduced} depends on the value of $h>0$.
  By Lemma \ref{lem:graph}, the curvature $\kappa_*$ of $\graph{U_*}$ satisfies that $\kappa_*<0$ and $(\kappa_*)_s<0$ on $(0,L_*)$, and hence we are able to use Corollary \ref{cor:monotoneangle} and thus deduce that $\omega_*$ is strictly decreasing.
  Since $U_*(1)=h_*$, we in particular have $\omega_*((0,L_*))=(-\theta_*,0)$, where $\theta_*:=\arctan{h_*}$ ($=1.03106...\approx 59.06^\circ$).
  The monotonicity of $\omega_*$ now implies that if $h<h_*$, then there is a unique solution $s\in(0,L_*)$ to \eqref{eq:reduced}, while if $h\geq h_*$, no solution exists.
  The proof is now complete.
\end{proof}

\begin{figure}
  \includegraphics[width=90mm]{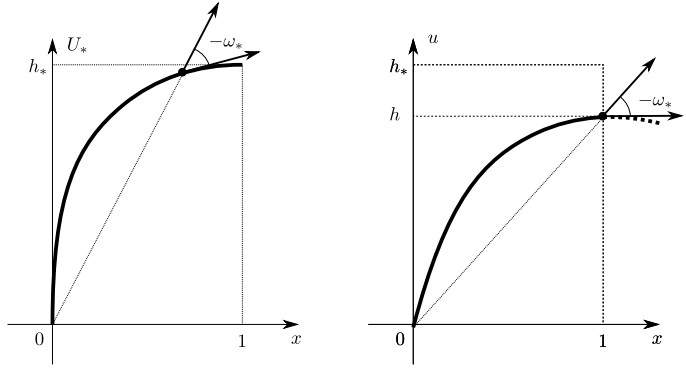}
  \caption{The correspondence between $U_*$ and $u$ in Lemma \ref{lem:1}}
  \label{fig:BVP}
\end{figure}

We are now in a position to complete the proof of Theorem \ref{thm:obstacle}.

\begin{proof}[Proof of Theorem \ref{thm:obstacle}]
  We first address the case of $h<h_*$.
  The existence of a minimizer in the symmetric class is already known (cf.\ \cite[Lemma 4.2]{DallAcqua2018}), so we only need to prove the uniqueness.
  If a function $\bar{u}\in X_{\psi,\mathrm{sym}}$ solves \eqref{eq:BVP}, then by \eqref{eq:apex}, symmetry, and $C^1$-regularity of $\bar{u}$, we deduce that the restriction $\bar{u}|_{[-1,0]}$ solves \eqref{eq:BVPsym} up to the shift $x\mapsto x+1$, and hence such a function $\bar{u}$ must be unique in view of Lemma \ref{lem:1}.

  We turn to the case of $h\geq h_*$.
  Nonexistence in the symmetric class $X_{\psi,\mathrm{sym}}$ directly follows by Lemma \ref{lem:1} as in the above proof.
  In the nonsymmetric case, we prove by contradiction, so suppose that a solution $\bar{u}\in X_{\psi}$ would exist.
  Then, up to reflection, $\bar{u}$ would take its maximum in $(-1,0]$.
  Letting $\bar{x}\in(-1,0]$ be a maximum point, and noting the scale invariance of free elasticae, we would deduce that the rescaled restriction
  $\tilde{u}(x):=\frac{1}{a}\bar{u}|_{[-1,\bar{x}]}(a(x+1)-1)$, where $a=\bar{x}-(-1)$, solves \eqref{eq:BVPsym} up to the shift as above and replacing $h$ with $\tilde{h}:=\frac{1}{a}\bar{u}(\bar{x})$.
  Then we would have $\tilde{h}\geq h_*$ since $a\leq1$ and $\bar{u}(\bar{x})=\max{\bar{u}}\geq\max{\psi}=h\geq h_*$, but this contradicts the nonexistence part of Lemma \ref{lem:1}.
  Therefore, no solution exists even in $X_{\psi}$.
\end{proof}

\begin{remark}[Regularity of solutions]
  By the nature of the variational inequality corresponding to our obstacle problem, the global regularity of a minimizer $u$ in \eqref{eq:min} is improved so that $u\in C^2(\bar{I})$ and $u'''\in BV(I)$, cf.\ \cite{DallAcqua2018,Dayrens2018}.
  The unique symmetric minimizer obtained here has this regularity, of course, but is not of class $C^3$ by the construction; this fact confirms optimality of the above regularity.
  Incidentally, we mention that our proof of uniqueness relies on just the trivial $C^1$-regularity at the apex of the cone, not invoking the higher regularity.
  However, we expect that the higher regularity would play an important role if we tackle the uniqueness problem in the general (nonsymmetric) class $X_{\psi}$ by our approach.
  We finally note that the $C^2$-regularity (or concavity etc.)\ is no longer true for an obstacle problem with an adhesion effect, cf.\ \cite{Miura2016,Miura2017}.
\end{remark}

\bibliography{bibliography}
\bibliographystyle{amsplain}

\end{document}